\documentclass[a4paper,reqno,12pt]{amsart}
\usepackage{amsmath}
\usepackage{amsthm}
\usepackage{amscd}
\usepackage{amssymb}
\usepackage{amsfonts}
\usepackage{latexsym}
\usepackage[mathscr]{eucal}

\newtheorem{thm}{Theorem}[section]
\newtheorem{prop}[thm]{Proposition}

\def\graph{\mathop{\rm {graph}}\nolimits}

\def\T{{\mathcal T}}

\def\cst{{C${}^*$}}

\begin{document}
\title{$C*$-algebras associated with complex dynamical systems and 
backward orbit structure}
\author{Tsuyoshi Kajiwara}
\address[Tsuyoshi Kajiwara]{Department of Environmental and
Mathematical Sciences,
Okayama University, Tsushima, 700-8530,  Japan}

\author{Yasuo Watatani}
\address[Yasuo Watatani]{Department of Mathematical Sciences,
Kyushu University, Motooka, Fukuoka, 819-0395, Japan}
\maketitle

\begin{abstract}
Let $R$ be a rational function. The iterations $(R^n)_n$ of $R$ gives 
a complex dynamical system on the Riemann sphere. We associate 
a $C^*$-algebra and study a relation between the $C^*$-algebra and 
the original complex dynamical system. 
In this short note, 
we recover the number of $n$-th backward orbits counted without 
multiplicity starting at 
branched points in terms of associated $C^*$-algebras with gauge 
actions. In particular, we can partially 
imagine how a branched point is moved to 
another branched point under the iteration of $R$. 
We use KMS states and a Perron-Frobenius type operator 
on the space of traces to show it.


\medskip\par\noindent
AMS SUBJECT CLASSIFICATION: 46L08, 46L55

\end{abstract}

\section
{Introduction}

Iteration of a rational function $R$ gives a complex dynamical 
system on the Riemann sphere $\hat{\mathbb C}$ . 
Since there exists a branched point (i.e. critical point),  
$R$ is {\it not} a local homeomorphism any more. 
Hence we are not able to introduce 
an \'{e}tale  groupoid in a usual way 
to associate a groupoid $C^*$-algebra of Renault \cite{Re}. 
For a branched covering $\pi : M \rightarrow M$, 
Deaconu and Muhly \cite{DM} introduced a $C^*$-algebra 
$C^*(M,\pi)$ as the $C^*$-algebra of the \'{e}tale 
groupoid  by substructing the branched points. 
 
In \cite{KW1}, we 
introduced slightly different $C*$-algebras by Cuntz-Pimsner construction 
to include branched points.   
Since the Riemann sphere $\hat{\mathbb C}$ is 
decomposed to the union of 
the Julia set $J_R$ and Fatou set $F_R$, 
we associated three $C^*$-algebras 
${\mathcal O}_R(\hat{\mathbb C})$,  
${\mathcal O}_R(J_R)$ and ${\mathcal O}_R(F_R)$
by considering $R$ as dynamical systems 
on $\hat{\mathbb C}$, $J_R$ and $F_R$ respectively. 
We have studied how properties of $R$ as 
complex dynamical systems are related with the 
structure of the associated 
$C^*$-algebras and their K-groups \cite{KW1}, \cite{IKW}, 
\cite{K} and \cite{W}. One of our aims is to 
analyze the singularity structure 
of the branched points in terms of operator algebras. For example, 
in \cite{IKW},  we showed that the extreme KMS states are 
parameterized by the branched points. 
Recently K. Thomsen introduces and studies 
another convolution $C^*$-algebra of the transformation groupoid 
adding local transfers 
for a rational function in \cite{Th1}, \cite{Th2}. 

In this short note, we study backward orbit structure 
in terms of operator algebras. In particular,  
we recover the number of $n$-th backward orbits counted 
without multiplicity starting at 
branched points in terms of associated $C^*$-algebras with gauge 
actions. If there exists a branched point in the backward orbits, 
then the number decreases  at the branched point, because we do not 
count the multiplicity. 
In particular, we can partially imagine how a branched point is move to 
another branched point under the iteration of $R$. 
We use KMS states and a Perron-Frobenius type operator to show it. 
We should mention that  Kumjian and Renault \cite{KR} study  the  existence and uniqueness of KMS 
states  associated to  general expansive maps which are local homeomorphisms. 

On the other hand, V. Nekrashevych \cite{Ne} 
studies the Cuntz-Pimsner
algebras for self-similar groups like iterated monodromy groups 
of expanding dynamical systems. 
Surprisingly,  he  reconstructed 
the complex dynamical system on the Julia set of 
a hyperbolic rational function from the Cuntz-Pimsner
algebra and the gauge action on it. 
Nekrashevych's  work does not include the case that 
the Julia set contains branched points, 
Thus our results
contains a new fact on the case.

Let $R$ be a rational function of the form 
$R(z) = \frac{P(z)}{Q(z)}$ with relatively prime polynomials 
$P$ and $Q$.   The degree of $R$ is denoted by 
$N = \deg R := \max \{ \deg P, \deg Q \}$. 
We regard a rational function $R$ as a $N$-fold branched 
covering map  $R : \hat{\mathbb C} \rightarrow \hat{\mathbb C}$
on the Riemann sphere $\hat{\mathbb C} = {\mathbb C} 
\cup \{ \infty \}$.  The sequence $(R^n)_n$ of iterations of $R$ 
gives a complex dynamical system on $\hat{\mathbb C}$. 
The Fatou set $F_R$ of $R$ is the maximal open subset of 
$\hat{\mathbb C}$ on which $(R^n)_n$ is equicontinuous (or 
a normal family), and the Julia set $J_R$ of $R$ is the 
complement of the Fatou set in $\hat{\mathbb C}$. 
The Fatou set $F_R$ is a stable part and  the Julia set $J_R$ 
is an unstable part. 

Recall that a {\it branched point} (or {\it critical point}) of $R$  
is a point   
$z_0$  at which $R$ is not locally one to one.  
It is a zero of $R'$ or 
a pole of $R$ of order two or higher.  The image $w_0 =R(z_0)$ 
is called a {\it branch value}  (or {\it critical value}) of $R$.    
Using appropriate local charts, if $R(z) = w_0 + c(z - z_0)^n + 
(\text{higher terms})$ with 
$n \geq 1$ and $c \not= 0$ on some neighborhood of $z_0$, 
then the integer $n = e(z_0) = e_R(z_0)$ is called the 
{\it branch index} of $R$ at $z_0$.  Thus $e(z_0) \geq 2$ if 
$z_0$ is a branched  point, and $e(z_0) = 1$ if $z_0$ is not. 
Therefore 
$R$ is an $e(z_0) :1$ map in a punctured neighborhood of $z_0$.
By the Riemann-Hurwitz formula, there exist $2N - 2$ branched points 
counted with multiplicity, that is, 
$\sum _{z \in \hat{\mathbb C}} (e(z) -1) = 2 \deg R -2$ .
Furthermore for each $w \in \hat{\mathbb C}$, we have 
$\sum _{z \in R^{-1}(w)} e(z) = \deg R $. Let $B_R$ be the 
set of branched points of $R$ and $C_R:= R(B_R)$ be the 
set of the critical values of $R$.  Then the restriction 
$R : \hat{\mathbb C} \setminus R^{-1}(C_R) \rightarrow 
\hat{\mathbb C} \setminus C_R$ is a $N :1$ regular covering, 
where $N = \deg R$. This means that any point $y \in 
\hat{\mathbb C} \setminus C_R $ has an open neighborhood $V$ 
such that $R^{-1}(V)$ has $N$ connected components $U_1, \dots, 
U_N$ and the restriction $R|_{U_k} : U_k \rightarrow V$ is a homeomorphism 
for $k = 1, \dots , N$.  Thus $R$ has $N$ analytic local cross 
sections $S_k = (R|_{U_k})^{-1}$. But if $y$ is in $C_R$, then 
there exist no such open neighborhood $V$. 
This fact causes many difficulties 
to analyze the associated $C^*$-algebra, since we include the 
branched points to construct the $C^*$-correspondence. If 
we will construct the associated groupoid naively, the  \'{e}taleness
( or r-discreteness)  
is not satisfied in general. This is the reason why we associated our  
$C^*$-algebras by Cuntz-Pimsner construction in \cite{KW1}. 
One of our aims is to 
analyze the singularity structure 
of the branched points in terms of operator algebras.  

In our previous paper \cite{IKW}, we study 
KMS-states for the gauge action based on Laca and Neshveyev \cite{LN}. 
The gauge action has a phase transition at $\beta = \log \deg R$. 
We can recover the degree of $R$, the number of branched points, 
the number of exceptional points and the orbits of exceptional points 
from the structure  of the KMS states. 
But we could not know  anything how branched points are related 
each other under the iteration of $R$. In this note we study 
the orbit structure of branched points under iteration. 
The proof depends on the fact that 
extreme KMS states are parameterized by the branched points as described in 
\cite{IKW}.

This work was supported by JSPS KAKENHI Grant Number 23540242, 19340040 
and 23654053.

\section{Construction of the associated $C^*$-algebras}

Since a rational function $R$ of degree at least two is {\it not} 
a homeomorphism, we can not use crossed product construction. 
We replace crossed pruduct construction  by 
Cuntz-Pimsner construction  to obtain the associated $C^*$-algebra. 

We recall Cuntz-Pimsner algebras \cite{Pi}.  
Let $A$ be a $C^*$-algebra 
and $X$ be a Hilbert right $A$-module.  We denote by $L(X)$ be 
the algebra of the adjointable bounded operators on $X$.  For 
$\xi$, $\eta \in X$, the "rank one" operator $\theta _{\xi,\eta}$
is defined by $\theta _{\xi,\eta}(\zeta) = \xi(\eta|\zeta)$
for $\zeta \in X$. The closure of the linear span of rank one 
operators is denoted by $K(X)$.   

A family $({u_i})_{i \in I}$ in $X$ is called a basis  
\cite{K} and \cite{KPW1}  ,  
(or a normalized tight frame more precisely as in \cite{FL}) of $X$ 
if 
$$
x = \sum _{i \in I} u_i(u_i|x)_A  \text{ for any } x \in X, 
$$ 
where the sum is taken as unconditional norm convergence, 
that is, for a directed set 
$\Lambda := \{F \ | \ F \subset I \text{ is a finite subset } \}$, 
$$
x = \lim_{F \in \Lambda} \sum_{i \in F} u_i(u_i|x)_A  
$$ 
Furthermore  $({u_i})_{i \in I}$ is called a finite basis if 
$(u_i)_{i \in I}$ is  a finite set. 
If a Hilbert $C^*$-module is countably generated, 
then there  exists a countable basis (that is, 
finite or a countably infinite basis ) of $X$ and written  
as $\{u_i\}_{i=1}^{\infty}$, where some $u_i$ may be zero. 
If $A$ has a unit and $X$ has a finite basis, then 
$X$ is algebraically finitely generated 
and projective over $A$ and $K(X) = L(X)$. 

We say that 
$X$ is a Hilbert $C^*$-bimodule (or $C^*$-correspondence) over $A$ 
if $X$ is a Hilbert right  $A$-
module with a homomorphism $\phi : A \rightarrow L(X)$.  
In this note, we assume  that 
$X$ is full and  $\phi$ is injective. 
Let $F(X) = \oplus _{n=0}^{\infty} X^{\otimes n}$
be the full Fock module of $X$ with the convention 
$X^{\otimes 0} = A$. 
 For $x \in X$, the creation operator
$T_x \in L(F(X))$ is defined by 
$$
T_x(a) =  x a  \qquad \text{and } \ 
T_{x}(x_1 \otimes \dots \otimes x_n) = x \otimes 
x_1 \otimes \dots \otimes x_n .
$$
We define $i_{F(X)}: A \rightarrow L(F(X))$ by 
$$
i_{F(X)}(a)(b) = ab \qquad \text{and } \ 
i_{F(X)}(a)(x_1 \otimes \dots \otimes x_n) = (\phi (a)
x_1) \otimes \dots \otimes x_n 
$$
for $a,b \in A$.  The Cuntz-Toeplitz algebra ${\mathcal T}_X$ 
is the $C^*$-subalgebra of $L(F(X))$ generated by $i_{F(X)}(a)$
with $a \in A$ and $T_x$ with $x \in X$.  
Let $j_K : K(X) \rightarrow {\mathcal T}_X$ be the homomorphism 
defined by $j_K(\theta _{x,y}) = T_xT_y^*$. 
We consider the ideal $I_X := \phi ^{-1}(K(X))$ of $A$. 
Let ${\mathcal J}_X$ be the ideal of ${\mathcal T}_X$ generated 
by $\{ i_{F(X)}(a) - (j_K \circ \phi)(a) ; a \in I_X\}$.  Then 
the Cuntz-Pimsner algebra ${\mathcal O}_X$ is the 
the quotient ${\mathcal T}_X/{\mathcal J}_X$ . 
Let $\pi : {\mathcal T}_X \rightarrow {\mathcal O}_X$ be the 
quotient map.  Put $S_x = \pi (T_x)$ and 
$i(a) = \pi (i_{F(X)}(a))$. Let
$i_K : K(X) \rightarrow {\mathcal O}_X$ be the homomorphism 
defined by $i_K(\theta _{x,y}) = S_xS_y^*$. Then 
$\pi((j_K \circ \phi)(a)) = (i_K \circ \phi)(a)$ for $a \in I_X$.   
We note that  the Cuntz-Pimsner algebra ${\mathcal O}_X$ is 
the universal $C^*$-algebra generated by $i(a)$ with $a \in A$ and 
$S_x$ with $x \in X$  satisfying that 
$i(a)S_x = S_{\phi (a)x}$, $S_xi(a) = S_{xa}$, 
$S_x^*S_y = i((x|y)_A)$ for $a \in A$, 
$x, y \in X$ and $i(a) = (i_K \circ \phi)(a)$ for $a \in I_X$.
We usually identify $i(a)$ with $a$ in $A$.  
If $X$ has a countable  basis $\{u_1,u_2, \dots\}$, then the last condition 
should  be replaced by 
$i(a) = \lim_{n \rightarrow \infty} \sum_{k =1}^n i(a)S_{u_k}S_{u_k}^*$ 
under the operator norm convergence for any $a \in I_X$. Since 
$\phi(a) \in K(X)$, we automatically have 
$\phi(a) = \lim_{n \rightarrow \infty} \sum_{k =1}^n
 \phi(a){\theta}_{u_k, u_k}$  
under the operator norm convergence, because  
$(\sum_{k =1}^n \phi(a){\theta}_{u_k, u_k})_n$ is an 
approximately units for $K(X)$.  

There exists an action 
$\gamma : {\mathbb R} \rightarrow Aut \ {\mathcal O}_X$
with $\gamma_t(S_{\xi}) = e^{it}S_{\xi}$, which is called the  
gauge action. Since we assume that $\phi: A \rightarrow L(X)$ is 
isometric, there is an embedding $\phi _n : L(X^{\otimes n})
 \rightarrow L(X^{\otimes n+1})$ with $\phi _n(T) = 
T \otimes id_X$ for $T \in L(X^{\otimes n})$ with the convention 
$\phi _0 = \phi : A \rightarrow L(X)$.  We denote by ${\mathcal F}_X$
the $C^*$-algebra generated by all $K(X^{\otimes n})$, $n \geq 0$ 
in the inductive limit algebra $\varinjlim L(X^{\otimes n})$. 
Let ${\mathcal F}_n$ be the $C^*$-subalgebra of ${\mathcal F}_X$ generated by 
$K(X^{\otimes k})$, $k = 0,1,\dots, n$, with the convention 
${\mathcal F}_0 = A = K(X^{\otimes 0})$.  Then  ${\mathcal F}_X = 
\varinjlim {\mathcal F}_n$. Consult \cite{Pi} and \cite{KPW2} for 
a general Cuntz-Pimsner algebras. .

Let $A= C(\hat{\mathbb C})$ and $X = C(\graph R)$ be the set of continuous 
functions on $\hat{\mathbb C}$ and $\graph R$ respectively, 
where $\graph R = \{(x,y) \in \hat{\mathbb C}^2 ; y = R(x)\} $ 
is the graph of $R$.
Then $X$ is an $A$-$A$ bimodule by 
$$
(a\cdot \xi \cdot b)(x,y) = a(x)\xi(x,y)b(y),\quad a,b \in A,\; 
\xi \in X.$$ 
We define an $A$-valued inner product $(\ |\ )_A$ on $X$ by 
$$
(\xi|\eta)_A(y) = \sum _{x \in R^{-1}(y)} e(x) \overline{\xi(x,y)}\eta(x,y),
\quad \xi,\eta \in X,\; y \in \hat{\mathbb C}.$$  
Thanks to the branch index $e(x)$, the inner product above gives a continuous  
function and $X$ is a full Hilbert bimodule over $A$ without completion. 
The left action of $A$ is unital and faithful. 

Since the Julia set $J_R$ is completely invariant under $R$, i.e., 
$R(J_R) = J_R = R^{-1}(J_R)$, we can consider the restriction 
$R|_{J_R} : J_R \rightarrow J_R$, which will be often denoted by 
the same letter $R$.
 Let $\graph R|_{J_R} = \{(x,y) \in J_R \times J_R \ ; \ y = R(x)\} $
be the graph of the restriction map $R|_{J_R}$ and 
$X(J_R) = C(\graph R|_{J_R})$. 
In the same way as above, $X(J_R)$ is a full Hilbert bimodule over $C(J_R)$. 
Since the Fatou set $F_R$ is also completely invariant, 
$X(F_R):=C_0(\graph R|_{F_R})$ is a full Hilbert bimodule over $C_0(F_R)$.

\medskip

\noindent
{\bf Definition}($C^*$-algebra associated with a complex dynamical system)
 Let $R$ be a rational function with $deg \ R \geq 2$. The $C^*$-algebra
${\mathcal O}_R(\hat{\mathbb C})$ is defined 
as the Cuntz-Pimsner algebra of the Hilbert bimodule 
$X= C(\graph R)$ over 
$A = C(\hat{\mathbb C})$. 
When the Julia set $J_R$ is not empty (for example 
$\deg R \geq 2$), we define 
the $C^*$-algebra ${\mathcal O}_R(J_R)$ 
as the Cuntz-Pimsner algebra of the Hilbert bimodule $
X= C(\graph R|_{J_R})$ over $A = C(J_R)$. 
When the Fatou set $F_R$ is not empty, the $C^*$-algebra
${\mathcal O}_R(F_R)$ is defined similarly.

\section
{Perron-Frobenius operator}

We shall introduce a Perron-Frobenius operator associated with a bimodule 
on the space of traces. 
Let $A$ be a unital $C^*$-algebra.  
We denote by $Trace(A)$ the set of bounded tracial functionals  
on $A$, 
$Trace^{+}(A)$ the set of bounded tracial positive functionals  
on $A$ and $Trace^+_1(A)$ the set of tracial states  
on $A$. We assume that $Trace(A)$ is not empty.
Let $X$ be a countably generated
Hilbert $A$-module and $\{u_i\}_{i=1}^{\infty}$ a countable 
basis of $X$.
For a tracial state $\tau$ on $A$, $\sup_n \sum_{i=1}^n
\tau((u_i|u_i)_A) \in [0,\infty]$ does not depend on the choice of basis
$\{u_i\}_{i=1}^{\infty}$ as in \cite{KW2} and \cite{KW3}. 
We put $d_{\tau} = \sup_n \sum_{i=1}^n
\tau((u_i|u_i)_A)$.  
We call that $X$ is of {\it finite degree type} if 
$\sup_{\tau \in Trace^+_1(A)}d_{\tau}< \infty$ 
(\cite {KPW1},  \cite{KW2} and \cite{KW3}). 
For example, 
let $X$ be the Hilbert bimodule  associated with a rational function $R$. 
Then $X_A$ is of finite degree type and 
$\sup_{\tau \in \T^+_1(A)}d_{\tau} = degree \ R$. 

{\bf Definition.}(Perron-Frobenius operator)
Let $A$ be a unital $C^*$-algebra and $X$  a countably generated (right) 
full Hilbert  module over $A$. 
Let $\phi: A \rightarrow L(X)$ be a unital faithful homomorphism so that 
$X$ is a bimodule over $A$. 
Let$\{u_i\}_{i=1}^{\infty}$ be a basis of $X$.
If $X$ is of finite degree type, then there exists a 
bounded linear operator $F_X: Trace(A) \rightarrow Trace(A)$ such 
that for $\tau \in Trace(A)$, 
$$
F_X(\tau)(a) = \sum_{i=1}^{\infty} \tau ((u_i \ |\ \phi(a)u_i)_A).
$$
Then $F_X$ does not depend on the choice of basis. 
We call $F_X$ a Perron-Frobenius operator associated with a bimodule $X$ 
of finite degree type. 
See \cite {KPW1}, \cite{KW2}, \cite{KW3} and \cite{IKW} for example. 

{\bf Example.}
Let $R$ be a rational function with $deg \ R \geq 2$.
Consider the $C^*$-algebra
${\mathcal O}_R(\hat{\mathbb C})$ 
associated with a complex dynamical system $(R^n)_n$ on 
the Riemann sphere. The $C^*$-algebra
${\mathcal O}_R(\hat{\mathbb C})$
is defined as the Cuntz-Pimsner algebra of the Hilbert bimodule 
$X= C(\graph R)$ over $A = C(\hat{\mathbb C})$. Then we shall show that 
the Perron-Frobenius operator $F_X$  associated with a bimodule $X$ 
is described as follows:   
For a finite Borel measure $\mu$ and 
$a \in A = C(\hat{\mathbb C})$, we have that
\[
(F_X(\mu))(a)= \mu (\tilde{a}).
\]
where a Borel function $\tilde{a}$ is defined by 
$\tilde{a}(y) = \sum_{x \in R^{-1}(y)} a(x)$ and we identify the finite 
Borel measure $\mu$ on $\hat{\mathbb C}$ with the  associated finite trace 
on  $C(\hat{\mathbb C})$ by the same symbol $\mu$.  
In particular, we have 
\[
F_X(\delta_y) = \sum_{x \in R^{-1}(y)}\delta_x,
\]
where $\delta_y$ is the Dirac measure on $y$. 
It is crucial that   
the sum on $x \in R^{-1}(y)$  should be  taken without multiplicity 
in these formulae. 

Let $\{u_{i}\}_{i=1}^{\infty}$ be a countable basis of $X$.
For $f \in X = C(graph \ R)$, we have 
$|f(x,y)| \leq \|f\|_2$ 
and the identity $\sum_{i=1}^{\infty}u_i(u_i|f)_A = f$ 
converges in norm $\|\ \|_2$. Hence the left side  
converges also pointwisely.
For each fixed $y \in \hat{\mathbb C}$ and $x \in R^{-1}(y)$,  
we consider the 
value of $\sum_{i=1}^{\infty}u_i(u_i|f)_A = f$  at $(x,y) = (x,R(x))$: 
\begin{align*}
& \lim_{n \to \infty} \left( \sum_{i=1}^{n}u_i(x,y)(u_i|f)_A(y)
\right) \\
= & \lim_{n \to \infty} \left( \sum_{i=1}^{n}u_i(x,y)
 \sum_{z \in R^{-1}(x)} e_R(z) \overline{u_i(z,y)}f(z,y) \right)
=  f(x,y). 
\end{align*}
We take $f \in X$ such that $f(x,y)=1$ and $f(x',y)=0$ for 
$x' \in R^{-1}(x)$ with $x' \not= x$. 
Then we have
\begin{align*}
&  \lim_{n \to \infty} \left( \sum_{i=1}^{n}u_i(x,y)
 \sum_{z \in R^{-1}(y)} e_R(z) \overline{u_i(z,y)}f(z,y) \right) \\
=  & \lim_{n \to \infty}\left( \sum_{i=1}^{n} e_R(x) \overline{u_i(x,y)}
u_i(x,y) \right) 
=  \sum_{i=1}^{\infty}e_R(x)|u_i(x,y)|^2=1.
\end{align*}
For $a \in A$, we have
\begin{align*}
 &  \sum_{i=1}^{\infty}(u_i|\phi(a)u_i)_A(y)
 = \sum_{i=1}^{\infty} \sum_ {x \in R^{-1}(y)} e_R(x) 
\overline{u_i(x,y)}
   a(x)u_i(x,y)  \\
 & =  \sum_{i=1}^{\infty} \sum_ {x \in R^{-1}(y)} e_R(x)
a(x)|u_i(x,y)|^2  \\
& =  \sum_ {x \in R^{-1}(y)}a(x) \left(\sum_{i=1}^{\infty}
    e_R(x) |u_i(x,y)|^2 \right)  \\
& = \sum_ {x \in R^{-1}(y)}a(x) 
 =  \tilde{a}(y).
\end{align*}

Therefore we have 
\[
(F_X(\mu)) (a) = \mu (\sum_{i=1}^{\infty}(u_i|\phi(a)u_i)_A) = \mu(\tilde{a}).
\]

\begin{prop}
Let $A$ be a unital $C^*$-algebra and $X$ a full 
Hilbert bi-module over $A$ with a unital faithful left action. 
 Let ${\mathcal O}_X$ be the 
Cuntz-Pimsner algebra for $X$ with a gauge action $\gamma$. 
Let $B = {\mathcal O}_X^{\gamma}$ be the fixed point algebra 
under $\alpha$. Let 
$$
Y = \{y \in {\mathcal O}_X \ | \ \gamma _z(y) = zy \ 
(z \in {\mathbb T}) \}
$$
be the 1-spectral subspace. Then we have the following: 

\begin{enumerate}
\item $Y$ is a Hilbert bi-module 
over $B$ under a natural action $b\cdot y \cdot c = byc$
with a $B$-valued inner product $(y\ | \ w)_B = y^*w$ 
for $b,c \in B$ and $y,w \in Y$. Moreover 
the linear span of $\{S_x b \ | \ b \in B, x \in X \}$ is dense in $Y$.  
\item If $\{u_i\}_{i=1}^{\infty}$ is a basis of $X$, 
then $\{S_{u_i}\}_{i=1}^{\infty}$ is a basis of $Y$. 
\item Let $F_Y:Trace(B) \rightarrow Trace(B)$ be the Perron-Frobenius 
operator 
associated with $Y$. Then $(F_Y(\tau))(a) = (F_X(\tau|_A))(a)$ 
for any trace $\tau \in Trace(B)$ and its restriction $\tau|_A$ 
to $A$ and $a \in A$. 
\end{enumerate}
\end{prop}
\begin{proof}
(1)Since  $y,w \in Y$ is in 1-spectral subspace, 
$y^*w$ is in the  
the fixed point algebra $B = {\mathcal O}_X^{\alpha}$
under $\alpha$. Since $\|(y\ | \ y)_B\| = \|y^*y\|$, 
$Y$ is complete. The others are also easily  checked. \\
(2)Let $\{u_i\}_{i=1}^{\infty}$ be a basis of $X$. 
For any $x \in X$, $b \in B$, we have 
\[
\sum _i S_{u_i}(S_{u_i}|S_xb)_B = \sum _i S_{u_i}S_{u_i}^* S_xb,  
=  S_{\sum _iu_i(u_i|x)_A}b = S_xb. 
\]
since $\|S_x\| = \|x\|$. \\ 
(3)For any trace $\tau \in Trace(B)$ and for any $a \in A$, 
\begin{align*}
(F_Y(\tau))(a) 
& = \sum_{i=1}^{\infty} \tau ((S_{u_i} \ |\ aS_{u_i})_B) 
= \sum_{i=1}^{\infty} \tau (S_{u_i}^* aS_{u_i}) \\
& = \sum_{i=1}^{\infty} \tau ((u_i \ |\ \phi(a)u_i)_A).
= (F_X(\tau|_A))(a).
\end{align*}
\end{proof}

{\bf Remark.} In the above, 
The Perron-Frobenius operator $F_X$ 
associated with a bimodule $X$ depends on the choice of the 
bimodule by definition. Since the bimodule $Y$ is defined by only the
$C^*$-algebra ${\mathcal O}_X$ with the gauge action $\gamma$, 
the Perron-Frobenius operator $F_Y$ associated with the bimodule $Y$ 
is an invariant of $C^*$-algebra ${\mathcal O}_X$ 
with the gauge action $\gamma$ up to conjugacy and does not depend on the 
original bi-module $X$. 
Moreover (3) of the above proposition 
shows that, for a fixed $\beta$, 
the set $Ext(KMS_{\beta}({\mathcal O}_X, \gamma))$ of extreme $\beta-KMS$ 
states for the gauge action $\gamma$ on ${\mathcal O}_X$ and 
the set 
$$
\{(F_Y(\tau|_B))(I) \ | 
\ \tau \in Ext(KMS_{\beta}({\mathcal O}_X, \gamma)) \}. 
 $$ 
are invariants of $C^*$-algebra ${\mathcal O}_X$ 
with the gauge action $\gamma$. We do study 
the number $(F_Y(\tau|_B))(I)= (F_X(\tau|_A))(I)$
and more generally  a sequence  
 $((F_Y^n(\tau|_B))(I))_n = ((F_X^n(\tau|_A))(I))_n$ 
in the next section.

\section
{orbit structure of branched points}

Since a rational functions is analytic as a map
of $\hat{\mathbb C}$ to  $\hat{\mathbb C}$ and 
have a rigid nature, its behavior on the singularities 
determines the main property of the rational function. 
Therefore it is important to study the orbit structure 
of the branched points of a rational function. 

\noindent
{\bf Definition.} Let $R$ be a rational function with $N =deg R \geq 2$. 
We denote by $b_n(z) = \ ^{\#}(R^{-n}(z)) $ 
the  number of the $n$-th backward orbit 
$R^{-n}(z)$ counted without multiplicity 
starting at $z \in \hat{\mathbb C}$. We define the 
associated sequence $b(z) := (b_n(z))_{n=0}^{\infty}$. If the backward orbit 
$\cup_{n=1}^{\infty} R^{-n}(z)$ has no intersection with the set 
$B_R$ of the branched points, then 
\[
b(z) = (1,N,N^2,N^3,\dots,N^n,\dots ). 
\]
In general the sequence $b(z)$ measures the existence of branched 
points in the  backward orbit $\cup_{n=1}^{\infty} R^{-n}(z)$ 
starting at $z$.

\noindent
{\bf Example.} Let $R(z) = z^2$. Then $B_R = \{0, \infty \}$. 
Since $R^{-1}(0) = \{0\}$ and $R^{-1}(\infty) = \{\infty \}$, 
we have
\[
b(0) =(1,1,1,1,1,\dots ), \ \ 
b(\infty) = (1,1,1,1,1,\dots )
\]

\medskip

\noindent
{\bf Example.} Let $R(z) = z^2+1$. Then $B_R = \{0, \infty \}$. 
Since there exist no branched point in the  backward orbit 
$\cup_{n=1}^{\infty} R^{-n}(0)$ 
starting at $0$ and $R^{-1}(\infty) = \{\infty \}$,
\[
b(0) =(1,2,4,8,16,\dots, 2^n,\dots ), \ \ \ 
b(\infty) = (1,1,1,1,1,\dots )
\]  

\medskip

\noindent
{\bf Example.} Let $R(z) = z^2-1$. Then $B_R = \{0, \infty \}$. 
Since $R(0) = -1, R(-1) = 0$ and $R^{-1}(\infty) = \{\infty \}$, 
we have that 
\[
b(0) =(1,2,3,6,11,\dots ) = (b_n(0))_n, \text{ where}
 \ b_{2n}(0) = \frac{1 + 2^{2n +1}}{3}, \ 
b_{2n+1}(0) = \frac{2 + 2^{2n +2}}{3}
\]
\[
b(\infty) = (1,1,1,1,1,\dots )
\]  

\medskip

\noindent
{\bf Example.} There exists a constant $c$ with $R(z) = z^2 + c$ 
such that $R^3(0) = 0$ and $R(0) \not= 0$, $R^2(0) \not=0$. 
Then $B_R = \{0, \infty \}$. 
We have that 
\[
b(0) =(1,2,4,7,14,\dots ), \ \ 
b(\infty) = (1,1,1,1,1,\dots )
\]

\noindent
{\bf Example.} For any fixed natural number $m \geq 2$, 
there exists a constant $c_m$ with $R(z) = z^2 + c_m$ 
such that $R^m(0) = 0$ and $R^k(0) \not= 0$ for $k = 1, \dots , m-1$, 
and  $B_R = \{0, \infty \}$. 
We have that 
\[
b(0) =(1,2,4,\dots,2^{m-1}, 2^m-1,\dots ), \ \ 
b(\infty) = (1,1,1,1,1,\dots )
\]
In fact, consider a sequence $(f_m)_m$ of  real functions defined by 
\[
f_{m+1}(x) = xf_m(x)^2 +1, \ \ \  f_1(x) = 1, \ \ f_2(x) = x + 1
\]
Then $f_m$ is a polynomial of degree $2^{m-1} -1$ and has a real root. 
Let 
\[
c_m := \min \{ x \in {\mathbb R} \ | \ f_m(x) = 0 \}
\]
Then $f_m(c_m) = 0$. We shall show that 
\[
c_2 = -1 > c_3 > \dots > c_m > c_{m+1} > \dots 
\]
Since $f_{m+1}(c_m) = c_mf_m(c_m)^2 +1 = 1 > 0$ and 
$f_{m+1}(x) \rightarrow -\infty$ as $x \rightarrow -\infty$, 
we have $c_{m+1} < c_m$. For $k = 1,2,\dots, m-1$, 
we have $f_k(c_m) \not= 0$, because $c_m < c_k$ and 
$c_k := \min \{ x \in {\mathbb R} \ | \ f_k(x) = 0 \}$.

Define $R(z) = z^2 + c$. Let $g_n(c)$ be the 
constant term of $n$-th iteration $R^n$ of $R$. 
Then $g_{n+1}(c) =  g_n(c)^2 + c$ and $g_1(c) = c$.
Then we have 
$g_n(c) = cf_n(c)$ by induction. 
Fix a natural number $m \geq 2$ and let $R(z) = z^2 + c_m$ in particular. 
Since $g_n(c_m)$ is the 
constant term of $n$-th iteration $R^n$ of $R$, 
$R^n(0) = g_n(c_m)$. Then $R^m(0) = g_m(c_m) = c_mf_m(c_m) = 0$. 
But for $k = 1,2,\dots, m-1$, 
we have $R^k(0) = g_k(c_m) = c_mf_k(c_m) \not= 0$.

\medskip

\noindent
{\bf Remark.} The main theorem bellow shows that we can distinguish 
these examples  
of quadratic polynomials in terms of $C^*$-algebras with gauge action, 
which could not be distinguished in our previous paper \cite{IKW}
where we counted only  the numbers of extreme $\beta$-KMS states.

\begin{thm}
Let $Q$ and $R$ be rational functions with the degrees at least two. 
Suppose that there exists an isomorphism 
$h: {\mathcal O}_{Q}(\hat{\mathbb C}) 
\rightarrow {\mathcal O}_{R}(\hat{\mathbb C})$ such that  
${\gamma}_R = h {\gamma}_Q h^{-1}$, where ${\gamma}_Q$ and 
${\gamma}_R$ are the associated guage actions. 
Then their backward orbit structures given by the 
number of $n$-th backward orbit starting at the 
branched points are same, that is, 
$$
\{b(z) \ |\  z \in B_Q \} = \{b(z) \ | \  z \in B_R \}
$$ 
\end{thm}
\begin{proof}
Suppose that there exists an isomorphism 
$h: {\mathcal O}_{Q}(\hat{\mathbb C}) 
\rightarrow {\mathcal O}_{R}(\hat{\mathbb C})$ such that  
${\gamma}_R = h {\gamma}_Q h^{-1}$. Then  the fixed point algebras by 
the gauge actions are isomorphic, which will be denoted by $B$. 
Moreover  the  1-spectral subspaces are isomorphic as Hilbert bi-module 
over $B$, which is denoted by $Y$. We should be careful that 
Hilbert bi-modules $X$  over 
the coefficient algebra 
$A = C(\hat{\mathbb C})$  are {\it not}  necessarily isomorphic. 
Therefore we should investigate invariants in terms of  
Hilbert bi-module $Y$ over $B$. 

We also note that  $deg \ R = deg\  Q =: N$, 
since the number of extreme $\beta$-KMS states   
$Ext(KMS_{\beta}({\mathcal O}_X, \gamma))$  for the gauge action 
$\gamma$ on ${\mathcal O}_X$ is exactly  $N = deg \ R$ for 
$\beta > deg \ R$, as in 
Theorem A in \cite{IKW}. 

Fix  $\beta > N = deg \  R$. Put $F_{X, \beta} = e^{-\beta}F_X$. For a 
branched point $z \in B_R$, let $\delta_z$ be the Dirac measure on 
$\hat{\mathbb C}$ corresponding to one point $z$. Define a trace 
$\tau_{\beta, z}$ on $A = C(\hat{\mathbb C})$, by 
\[
\tau_{\beta,z} = 
 m_{\beta,z}\sum_{k=0}^{\infty} F_{X, \beta}^k( \delta_z) 
= m_{\beta,z}\sum_{k=0}^{\infty} \frac{1}{e^{k\beta}} 
 \sum_{x \in R^{-k}(z)} \delta_x, 
\]
where $m_{\beta,z}$ is the normalized constant and given  by 
\[
m_{\beta,z} = (\sum_{k=0}^{\infty} \frac{1}{e^{k\beta}} 
 \sum_{x \in R^{-k}(z)} 1  )^{-1}
= (\sum_{k=0}^{\infty} \frac{1}{e^{k\beta}} 
  b_k(z))^{-1}. 
\]
Let $E: {\mathcal O}_{R}(\hat{\mathbb C})  \rightarrow 
{\mathcal O}_{R}(\hat{\mathbb C})^{\gamma}$ be the  
conditional expectation on to the fixed point algebra 
${\mathcal O}_{R}(\hat{\mathbb C})^{\gamma}$ by the gauge 
action defined by 
$E(T) = \int _{\mathbb T} \gamma _z(T)dz$. By 
Theorem A in \cite{IKW}, 
there exists a unique $\beta$-KMS state $\varphi{\beta,z}$ 
on ${\mathcal O}_{R}(\hat{\mathbb C})$ such that 
its  restriction to $A = C(\hat{\mathbb C})$ is 
exactly $\tau_{\beta, z}$. 
Moreover the set of extreme $\beta$-KMS states has a bijective 
correspondence to the set $B_R$ of the branched points under 
the correspondence between $\varphi_{\beta,z}$ and $z \in B_R$. 
The state satisfies that 
$\varphi_{\beta,z} = \varphi_{\beta,z} \circ E$ and 
\[
\varphi_{\beta,z}
(S_{x_1}S_{x_2}\dots S_{x_n}S_{y_n}^*\dots S_{y_2}^*S_{y_1}^*) 
= e^{-\beta} 
\tau_{\beta, z}((y_1 \otimes y_2 \dots \otimes y_n \ 
| \ x_1 \otimes x_2 \dots \otimes x_n)_A).
\]

Consider the fixed point algebra 
$B = {\mathcal O}_{R}(\hat{\mathbb C})^{\gamma}$ by the gauge action 
$\gamma$. Let $F_Y: Trace(B) \rightarrow Trace(B)$ be the 
Perron-Frobenius operator associated with the 1-spectral subspace
$$
Y = \{y \in {\mathcal O}_X \ | \ \gamma _t(y) = ty \ 
(t \in {\mathbb T}) \}. 
$$
Define a sequence $c(z) = (c_n(z))_n$ by 
\[
c_n(z) = (F_{Y, \beta}^n(\varphi_{\beta,z}|_B))(1)
 = (F_{X, \beta}^n(\tau_{\beta, z}))(1), 
\]
which depends on only 
$({\mathcal O}_{R}(\hat{\mathbb C}),{\mathbb T},  \gamma_R)$ 
and  $\varphi_{\beta,z}$. Therefore 
 the family  $\{c(z) \ | \  z \in B_R \}$
of such sequences depends only on 
$({\mathcal O}_{R}(\hat{\mathbb C}),{\mathbb T},  \gamma_R)$ 
up to conjugacy. To make  the proof finished, it is enough to 
show that the  number  $b_n(z)$ of the $n$-th backward orbit 
$R^{-n}(z)$ starting at $z \in B_R$ is described in terms of 
the sequence $c(z) = (c_n(z))_n$. 

Since 
\[
\tau_{\beta, z} - F_{X, \beta}(\tau_{\beta, z}) = m_{\beta,z}\delta_z, 
\]
we have 
\[
1 - c_1(z) = (\tau_{\beta, z} - F_{X, \beta}(\tau_{\beta, z}))(1) 
= m_{\beta,z}\delta_z(1) =  m_{\beta,z}. 
\]
In general, since 
\[
 F_{X, \beta}^n(\tau_{\beta, z}) - F_{X, \beta}^{n+1}(\tau_{\beta, z}) 
= m_{\beta,z}F_{X, \beta}^n(\delta_z), 
\]
we have 
\[
c_n(z) - c_{n+1}(z) 
=  (F_{X, \beta}^n(\tau_{\beta, z}) - F_{X, \beta}^{n+1}(\tau_{\beta, z}))(1) 
= m_{\beta,z}F_{X, \beta}^n(\delta_z)(1) = m_{\beta,z}b_n(z)
\]
Hence
\[
b_n(z) = \frac{c_n(z) - c_{n+1}(z)}{1-c_1(z)}.  
\]

\end{proof}

By a similar argument we have a theorem on the $C^*$-algebra
${\mathcal O}_{R}(J_R)$ associated with a complex dynamical system
$(R^n)_n$ restricted  
to  the Julia set $J_R$. 

\begin{thm}
Let $Q$ and $R$ be rational functions with the degrees at least two. 
Suppose that there exists an isomorphism 
$h: {\mathcal O}_{Q}(J_Q) 
\rightarrow {\mathcal O}_{R}(J_R)$ such that  
${\gamma}_R = h {\gamma}_Q h^{-1}$, where ${\gamma}_Q$ and 
${\gamma}_R$ are the associated guage actions. 
Then their backward orbit structures 
given by the 
number of $n$-th backward orbit 
starting at the 
branched points on the Julia sets are same, that is, 
$$
\{b(z) \ |\  z \in B_Q \cap J_Q \} 
= \{b(z) \ | \  z \in B_R \cap J_R \}
$$ 
\end{thm}

\end{document}